\documentclass[11pt,fleqn]{amsart}
\usepackage{latexsym,amsmath,amssymb,amsfonts,MnSymbol,epsfig}
\usepackage[all]{xy}

\providecommand{\CC}{{\mathbb{C}}}
\providecommand{\RR}{{\mathbb{R}}}

\providecommand{\QT}{{\mathbb{H}}} 
\providecommand{\OT}{{\mathbb{O}}} 
\providecommand{\KD}{{\mathbb{K}}} 

\providecommand{\im}{{\rm Im}} 

\providecommand{\DD}{{\mathcal D}} 

\providecommand{\KK}{{\mathcal K}}

\providecommand{\Lg}{{\mathfrak g}}

\newcommand{\ang}[1]{\langle #1 \rangle} 

\newtheorem{definition}{Definition}
\newtheorem{lemma}[definition]{Lemma}
\newtheorem{proposition}[definition]{Proposition}
\newtheorem{corollary}[definition]{Corollary}
\newtheorem{theorem}[definition]{Theorem}

\begin{document}

\title{Contact Structures of Arbitrary Codimension and Idempotents in the Heisenberg Algebra}
\author{Erik van Erp}
\address{Dartmouth College, 6188, Kemeny Hall, Hanover, New Hampshire, 03755, USA}
\curraddr{University of Hawaii at Manoa, 2565 McCarthy Mall, Honolulu, Hawaii 96822, USA}
\email{jhamvanerp@gmail.com}

\maketitle


\begin{abstract}
A contact manifold is a manifold equipped with a distribution of codimension one that satisfies a ``maximal non-integrability'' condition. 
A standard example of a contact structure is a strictly pseudoconvex CR manifold,
and operators of analytic interest are the tangential Cauchy-Riemann operator $\bar{\partial}_b$ and the Szeg\"o projector onto the kernel of $\bar{\partial}_b$. The Heisenberg calculus is the pseudodifferential calculus originally developed  for the analysis of these operators.

We introduce a ``non-integrability'' condition for a distribution of arbitrary codimension that directly generalizes the definition of a contact structure. 
We call such distributions polycontact structures.
We prove that the polycontact condition is equivalent to the existence of nontrivial projections in the Heisenberg calculus, and show that generalized Szeg\"o projections exist on polycontact manifolds.
We explore some geometrically interesting examples of polycontact structures.
\end{abstract}

\section{Introduction}

Geometers have considered various ways to define contact structures in codimension $3$.
This paper is a contribution to this consideration from an analytic perspective.
We propose here a notion of ``contact structure'' for distributions of {\em arbitrary} codimension.
Roughly speaking, a {\em polycontact structure} is a distribution that is ``maximally twisted'' or ``maximally non-integrable'' (the precise definition is in section \ref{section:def}).
Our definition is a straightforward generalization of the usual notion to higher codimensions.
It is remarkable that this simple-minded generalization is equivalent to a non-trivial analytic property. 

\begin{theorem}\label{thm:main}
Let $M$ be a compact connected manifold with distribution $H\subseteq TM$.
Then there exist projections in the algebra of Heisenberg pseudodifferential operators on $M$
that have infinite dimensional kernel and range
if and only if $H$ is polycontact.
\end{theorem} 
Idempotents in the algebra of {\em classical} scalar pseudodifferential operators are either smoothing operators, or of the form $1-x$ where $x$ is smoothing. 
They have either finite dimensional range or finite dimensional kernel. 
But as is well-known, on contact manifolds there exist generalized Szeg\"o projections, which have infinite dimensional kernel and cokernel.
These projections can be  constructed as Heisenberg pseudodifferential operators (see [EM03]).
The construction of generalized Szeg\"o projections generalizes rather easily to polycontact manifolds, and Theorem \ref{thm:main} implies that polycontact manifolds are precisely those for which generalized Szeg\"o projections can be defined in the Heisenberg calculus.
Given this striking analytic fact, we suspect that these structures must have interesting geometric properties as well.

In various ways, polycontact structures are geometrically similar to contact structures.
The Heisenberg group is a standard example of a contact manifold, and nilpotent groups of Heisenberg type are examples of polycontact structures,
and they exist in any codimension.
An observation due to Sean Fitzpatrick is that a distribution $H$ is polycontact iff the normal bundle $N^*\setminus 0$ is a symplectic submanifold of $T^*M$. 
The fact that polycontact manifolds have generalized Szeg\"o projections as well as a symplectization suggests that they fit very nicely in the framework of {\em Toeplitz structures} developed by Boutet de Monvel and Guillemin [BG81],
a connection we will explore elsewhere.
Biquard's {\em quaternionic contact manifolds} are polycontact, and we identify the extra geometric structure (a compatible conformal structure in the distribution $H$) that characterizes quaternionic contact structures among general $3$-polycontact structures.
We suspect that {\em conformal polycontact manifolds} share many of the geometric properties of quaternionic contact manifolds.
It is also interesting that any strictly pseudoconvex real hypersurface in a hypercomplex manifold is $3$-polycontact, while in general it may not be quaternionic contact.
Another example of polycontact structures are the {\em fat bundles} introduced by Alan Weinstein. Roughly speaking, fat bundles have a lot of curvature, as opposed to flat bundles which have no curvature. For example, a bundle with structure group $S^1$ is fat if and only if its curvature defines a symplectic form on the base (which of course means that the total space is contact). We will see that, in general, a connection on a principal bundle is fat if and only if its horizontal distribution is polycontact.

A result of independent interest is a characterization of polycontact structures in terms of the structure theory of the $C^*$-algebra $\Psi_H^*$---the closure in operator norm of the algebra $\Psi^0_H(M)$ of order zero Heisenberg pseudodifferential operators.
We show that for a compact polycontact manifold the $C^*$-algebra $\Psi^*_H$ has a decomposition series consisting of only two sequences,
\[ 0\to I_H^*\to \Psi_H^*\to C(S^*H)\to 0\]
\[ 0\to \KK\to I_H^*\to C(S^*N, \KK)\to 0\]
In [Dy78], Alexander Dynin announced this result for contact manifolds,
and we prove it here for all polycontact manifolds.
Thus, the polycontact condition is equivalent to a simple structural feature of the $C^*$-algebra $\Psi^*_H$.

\vskip 6pt \noindent  {\bf Acknowledgements.}
I thank Sean Fitzpatrick for carefully reading the manuscript, catching some errors,  and suggesting some interesting geometric examples and properties of polycontact structures.
The term polycontact structure was suggested by Alan Weinstein.    
I especially thank the anonymous referee for numerous suggestions that greatly improved the exposition of the paper.

\section{Polycontact Structures}\label{section:def}

A contact form on a manifold $M$ of dimension $2m+1$ ($m\ge 1$) is a $1$-form $\theta$ such that $\theta(d\theta)^m$ is a nowhere vanishing volume form on $M$.
A distribution $H\subset TM$ defines a contact structure on $M$ if it is the kernel of a contact form, at least locally.
We generalize this definition to higher codimensions as follows.
\begin{definition}\label{def1}
A {\em polycontact structure} on a smooth manifold $M$ 
is a distribution $H\subset TM$ (of corank at least $1$)
such that if $\theta$ is any nonzero one-form that vanishes on $H$ ($\theta$ may be defined only locally), then the restriction of $d\theta$ to $H$ is nondegenerate.
\end{definition}
Associated to any distribution $H$ is an anti-symmetric bilinear bundle map
\[ B\;\colon\; H\times H\to N\]
induced by Lie brackets of vector fields, i.e.,
\[ B_m(X(m), Y(m)) = [X,Y](m)\qquad {\rm mod}\,H_m\]
where $m\in M$ and $X,Y$ are two vector fields in $H$.
We will refer to this map as the {\em Levi form} of $H$. 
It is also sometimes referred to as the ``curvature'' or ``twistedness'' of $H$.
The following lemma says that $H$ is polycontact if it is ``maximally twisted''.

\begin{lemma}\label{lemma:equiv}
The following are equivalent:
\begin{enumerate}
\item $H$ defines a polycontact structure.
\item For every $m\in M$ and $\theta\in N_m^*\setminus \{0\}$ the bilinear from $\theta\circ B_m\colon H_m\times H_m\to \RR$ is non-degenerate.
\item The dual Levi form
\[ B^*\;\colon\; N^*\to \Lambda^2 H^*, \quad B_m^*(\theta)=\theta\circ B_m, \;\theta\in N_m^*\]
is injective and maps non-zero vectors in $N_m^*$ to symplectic forms on $H_m$.
\end{enumerate}
\end{lemma}
\begin{proof}
From $d\theta(X,Y) = X.\theta(Y) - Y.\theta(X)-\theta([X,Y])$ it follows that for two vector fields $X, Y$ in $H$ we have $d\theta(X,Y) = -\theta([X,Y])$.
Therefore $\theta\circ B_m$ is just the restiction of $-d\theta$ to $H_m$.

\end{proof}

As we shall see, polycontact structures are characterized by the existence of generalized Szeg\"o projections.
From a purely geometric point of view it is interesting that polycontact structures are characterized by the existence of a canonical symplectization. The following observation is due to Sean Fitzpatrick and the proof given here was suggested by Alan Weinstein [Fi11].

\begin{proposition}
Let $H\subset TM$ be a distribution of corank $p\ge 1$.
Then $N^*\setminus 0$ is a symplectic subspace of $T^*M$ (with respect to the canonical symplectic structure of $T^*M$)
if and only if $H$ defines a polycontact structure.
\end{proposition}
\begin{proof}
Let $X_1,\dots, X_q$ be a local frame for the vector bundle $H$.
The corresponding functions $f_i(m,\theta)=\theta(X_i(m))$ for $(m,\theta)\in T^*M$ are defining functions for $N^*\setminus 0\subset T^*M\setminus 0$.
In general, a submanifold of a symplectic manifold with defining function $f_1,\dots, f_q$ is symplectic if and only if the matrix $A_{ij}=\{f_i, f_j\}$ of Poisson brackets is nondegenerate at each point.
Because $\{f_X,f_Y\}=f_{[X,Y]}$, we see that here
\[ A_{ij}(m,\theta)=\{f_i,f_j\}(m,\theta)=\theta([X_i,X_j](m))=\theta\circ B_m(X_i(m),X_j(m))\]
Thus, the matrix $A(m,\theta)$ is nondegenerate if an only if $\theta\circ B_m$ is.
\end{proof}

The local model for contact manifolds is the Heisenberg group with its standard translation invariant contact structure. 
Likewise, polycontact structures are intimately related to certain two-step nilpotent groups. However, for polycontact manifolds nilpotent groups are only `microlocal' models, not necessarily local models as in the case of corank $1$. 

The Levi form can be interpreted as a Lie bracket on $\Lg_m=H_m\times N_m$ (where $N_m$ is central in $\Lg_m$), making $\Lg_m$ into a graded two-step nilpotent Lie algebra.
The corresponding group structure on $G_m=H_m\times N_n$ is defined by
\[ (h, n)(h',n')=(h+h', n+n'+\frac{1}{2}B_m(h,h')),\quad h,h'\in H_m,\; n,n'\in N_m.\]
We refer to the group $G_m$ as the {\em tangent group} of the distribution $H$ at $m\in M$.

On $G_m$ consider the left translation invariant distribution corresponding to the degree 1 subspace $H_m\subset \Lg_m=T_0G_m$.
This distribution is polycontact iff the bracket $B_m$ satisfies the conditions of Lemma \ref{lemma:equiv}.
In this case we will call the group $G_m$ a {\em polycontact group}.
The distributions on the tangent groups $G_m$ are microlocal approximations of the distribution on the manifold $M$, and $(M,H)$ is a polycontact manifold iff all its tangent groups are polycontact.
In corank $1$ the Heisenberg group is the only polycontact group,
but in higher coranks the tangent group may vary from point to point.

Finally, observe that a polycontact structure satisfies the much weaker condition that the bracket $B_m$ is {\em onto}, and therefore  $H+[H,H]=TM$.
Thus $H$ satisfies the H\"ormander bracket condition, which is an important condition in sub-Riemannian geometry.

\section{Examples}

In this section we discuss several examples of polycontact structures.
The natural place to start is to look for nilpotent groups with translation invariant polycontact structure.
Groups of Heisenberg type provide a large class of examples, proving that polycontact structures exist in any codimension.
We then consider geometric structures on manifolds that are closely related to polycontact structures: fat connections, hypersurfaces in quaternionic domains, quaternionic contact manifolds, and polycontact CR structures.

\subsection{Heisenberg Type Groups}\label{s1} 


A nice class of examples of polycontact groups are the {\em Heisenberg type groups} introduced by Aroldo Kaplan in [Ka80].
Kaplan defined these groups in connection with the study of fundamental solutions of sublaplacians on nilpotent groups.
Here we describe examples of Heisenberg type groups with center of arbitrary dimension $p$. These examples illustrate that polycontact structures exist in any codimension.
For a general discussion of Heisenberg type groups and their properties, see [CDR91, Ka80, Ka83].

\begin{definition}\label{H-type}
A simply connected, connected two-step nilpotent Lie group $G=V\times W$ is of {\em Heisenberg type}, or of {\em H-type}, if $[V,V]=W$ and if $V$ and $W$ can be equipped with positive definite inner products such that for every $w\in W$ with $\|w\|=1$ the endomorphism $J_w\in {\rm End}(V)$ defined by
\[ \langle J_w v, v'\rangle = \langle w,[v,v']\rangle\]
is orthogonal. 
\end{definition}
Clearly every Heisenberg type group is polycontact, since $\langle v,[w,w']\rangle = \theta\circ B(w,w')$ with $\theta=\langle v,-\rangle\in W^*$.

Observe that since the Lie bracket is antisymmetric we must have $J_v^T=-J_v$, and so orthogonality of $J_v$ amounts to $J_v^2=-1$.
Definition \ref{H-type} can be made more explicit as follows.
Choosing an orthonormal basis for $W=\RR^p$ amounts to choosing $p$ endomorphisms $J_1,\dots, J_p\in {\rm End}(\RR^q)$ that satisfy the relations
\[ J_k^2=-1,\quad J_k^T=-J_k,\quad J_kJ_l = -J_lJ_k,\quad \forall k\; \forall l\ne k\]
If we let $\omega_k(v,w)=\ang{J_kv, w}$ be the symplectic form associated to $J_k$,
then the bracket of a Heisenberg type Lie algebra is obtained by letting
\[ [v,w] = \left( \,\omega_1(v, w),\, \omega_2(v,w),\, \dots,\, \omega_p(v,w)\, \right) \in \RR^p\,\quad v,w\in \RR^q\]
Note that every {\em polycontact} group can be presented in this way, by choosing $p$ symplectic forms $\omega_1,\omega_2,\dots,\omega_p$ on $\RR^q$.
However, it is not sufficient to choose $p$ linearly independent forms.
Linear independence merely guarantees that any nontrivial linear combination $c_1\omega_1+\cdots +c_p\omega_p$ is non-zero. 
We need to satisfy the much stronger condition that all nontrivial linear combinations $c_1\omega_1+\cdots +c_p\omega_p$ are nondegenerate.

For this reason it is convenient to work with the complex structures $J_k$ instead of the symplectic forms $\omega_k$. 
Choosing orthogonal complex structures $J_k$ that satisfy the desired anticommutation relations amounts to choosing a representation of the {\em Clifford algebra} $Cl(\RR^p)$ on the vector space $\RR^q$.
The Clifford algebra $Cl(\RR^p)$ is the universal (real) algebra generated by vectors $w\in \RR^p$ with relations $w^2 = -\|w\|^2$.
To find $p$ endomorphisms $J_1,\dots,J_p$ with the desired properties,
choose a representation $c\;\colon Cl(\RR^p)\to {\rm End}(\RR^m)$ of the Clifford algebra.
A representation exists for every value of $p$,
and depending on the value of $p$ there exist only one or two {\em irreducible} representations, up to equivalence.
The dimension $m$ of any irreducible representation of $Cl(\RR^p)$ is the smallest number of the form $m=2^{4s+t}$ with $t=0,1,2,3$ such that $p<8s+2^t$. 
In fact, the space $\Lambda^2\RR^q$ contains a $p$-dimensional subspace of symplectic forms iff $q\ge m$ (a result due to Hurwitz, Radon, and Eckmann; see [Hu66] Ch.11, Thm.8.2).
Thus, the irreducible Clifford algebra representations realize the minimal possible value $m=q$, and any $p$-polycontact manifold must at least be of dimension $p+m$.

Now let $J_1,\dots J_p\in {\rm End}(\RR^m)$ be the orthogonal anticommuting complex structures on $\RR^m$ obtained from an irreducible representation of $Cl(\RR^p)$.
Identify $\RR^{mn}=(\RR^m)^n$ by using coordinates $v=(v_1, \dots, v_n)$ with $v_i\in \RR^m$. 
Then we denote by
\[ {\mathcal C}(p, n) =\RR^{mn}\times \RR^p\]
the Heisenberg type group whose bracket is defined by the relation 
\[ [v,w] = \sum_{i=1}^n \left(\,\ang{J_1v_i, w_i},\, \ang{J_2v_i, w_i},\dots,\,\ang{J_pv_i, w_i}\,\right)\,\in \RR^p,\quad v,w\in (\RR^m)^n\]
If $p=1$ the only irreducible representation of $Cl(\RR)=\CC$ is $\RR^2$ with a choice of complex structure $J^2=-1$. 
Identifying $\RR^2=\CC$, $J=i$, we see that $\ang{v,Jw}= \im \,v\bar{w}$,
and the group ${\mathcal C}(1,n)$ is just the Heisenberg group with group operation
\[ (v,w)(v',w')=(v+v', w+w'+\frac{1}{2}\im \left(\sum_{i=1}^n v_i\bar{v'}_i\right)),\qquad v,v'\in \CC^n,\;w,w'\in \RR.\]
The same formula defines the group operation for the quaternionic Heisenberg group $\QT^n\times \im \QT$ and the octonionic Heisenberg group $\OT^n\times \im \OT$.
In fact, we have isomorphisms
\[ {\mathcal C}(1,n)\cong \CC^n\times \im \CC,\quad {\mathcal C}(3,n)\cong \QT^n\times \im \QT,\quad {\mathcal C}(7,n)\cong \OT^n\times \im \OT.\]
Recall that the two non-equivalent irreducible representations of the Clifford algebra $Cl(\RR^3)$ on $\RR^4$ are obtained by identifying $\RR^3=\im \QT$ and then letting $\im \QT$ act by either left multiplication on $\RR^4=\QT$, or by right multiplication with the negative vector.
The groups ${\mathcal C}(3,n)$ that result from either choice of representation are isomorphic.
To see that it is the quaternionic Heisenberg group one only needs to verify that
\[ v\bar{w} = \ang{v, w} + \ang{v, iw}\, i+\ang{v, jw}\, j+\ang{v, kw}\, k,\quad v,w\in \QT\]
To see that ${\mathcal C}(7, n)\cong \OT^n\times \im \OT$ we identify the irreducible representation space $\RR^8$ of the Clifford algebra $Cl(\RR^7)$ with the octonians.
As before, one lets vectors in $\RR^7=\im \OT$ act by (left or right) octonionic multiplication on $\RR^8=\OT$.
While octonion multiplication is not associative, it is {\em bi-associative}, i.e., any sub-algebra generated by two elements is associative (because it is at most four dimensional).
Therefore the desired equality $w(wv)=(w^2)v=-v$ holds for imaginary octonions $w\in \im \OT$ of unit length, because then  $w^2=-\|w\|^2$.
(For a nice discussion of octonions, see [Ba02, CS03].)

The codimension of a polycontact structure does not have to be odd. 
For example, with $\KD=\RR, \CC, \QT, \OT$ let
\[ G=\KD^{2n}\times \KD.\]
Then a two-step nilpotent Lie bracket can be defined by
a kind of generalized `symplectic form' $\omega$, given by
\[ \omega\;\colon \KD^{2n}\times \KD^{2n}\to \KD\;;\; \omega(u,v) = \sum_{i=1}^n\,u_i\,\bar{v}_{i+n}-v_i\,\bar{u}_{i+n}.\]
If $\KD=\RR$ then $\omega$ is the standard symplectic form on $\RR^{2n}$,
and we recover the Heisenberg group.
The groups $\KD^{2n}\times \KD$ are of Heisenberg type, and are therefore examples of nilpotent groups with polycontact structures of codimension $1, 2, 4, 8$ respectively.
We leave it as an exercise to verify that
\[
{\mathcal C}(2, n) \cong \CC^{2n}\oplus \CC,\quad
{\mathcal C}(4, n) \cong \QT^{2n}\oplus \QT,\quad {\mathcal C}(8, n) \cong \OT^{2n}\oplus \OT.
\]


\vskip 6pt
\noindent {\bf Remark.}
Every Heisenberg type group with center of dimension $p$
is associated to a (finite dimensional) representation of the Clifford algebra $Cl(\RR^p)$,
and equivalent representations determine isomorphic groups.
Clifford algebra representations are completely reducible as a direct sum of irreducible representations, and depending on the value of $p$ there are either $1$ or $2$ non-equivalent irreducible representations.
So given the dimension of the group and of the center there are only finitely many Heisenberg type groups.
If $p$ is not equal $3$ modulo $4$ there is only one irreducible representation of $Cl(\RR^p)$, 
and the groups ${\mathcal C}(p,n)$ are the only Heisenberg type groups.
If $p=3$ mod $4$ there are {\em two} nonequivalent irreducible representations of $Cl(\RR^p)$, and the isomorphism class of the corresponding Heisenberg type group $G=\RR^{mn}\times \RR^p$ depends on the type of representation. 
It is not hard to see that the two possible {\em isotypical} representations of $Cl(\RR^p)$ on $\RR^{mn}$ determine isomorphic groups.
These are the groups $\mathcal{C}(p,n)$ discussed here. 
(For the representation theory of Clifford algebras see [Hu66, LM89].). 
\vskip 6pt

At the end of the paper in which Kaplan first introduced groups of Heisenberg type he briefly considered a larger class of two-step nilpotent groups (condition (18) in  [Ka80]) that is easily seen to be equivalent to the class of polycontact groups.
Kaplan's remarks show that polycontact groups are precisely those (connected, simply connected) two-step nilpotent Lie groups for which the sublaplacian is {\em analytically} hypoelliptic.

It is interesting to identify what distinguishes Heisenberg type groups among  all polycontact groups. 
This will help us identify, for example, the difference between a quaternionic contact manifold (modeled on ${\mathcal C}(3,n)\cong \QT^n\times \im \QT$) and a general $3$-polycontact manifold.

\begin{proposition}\label{prop:conf}
Let $G=V\times W$ be a two-step connected, simply-connected nilpotent group. 
Then $G$ is of Heisenberg type if and only if it is polycontact and there exists a conformal structure on $V$ that is compatible with the symplectic forms $\omega_\theta(v,v')=\theta([v,v'])$ for all nonzero $\theta\in W^*\setminus \{0\}$.
\end{proposition}
\noindent {\bf Remark.} A symplectic form $\omega$ is compatible with a positive quadratic form $q$ on a vectorspace $V$ if there exists a complex structure $J$ on $V$ ($J^2=-1$) that is orthogonal with respect to $q$ and such that $\omega(v,v')=q(v,Jv')$.
A symplectic form is compatible with a conformal structure if there exists a quadratic form compatible with both.
\vskip 6pt
\begin{proof}
It is obvious from the definition that Heisenberg type groups have a compatible conformal structure in $V$. We need to show the converse.
Suppose that $G$ is polycontact and that $V$ has a compatible conformal structure.
Choose a corresponding positive-definite quadratic form $q$ on $V$
and let $J_\theta\in {\rm End}(V)$ be such that
\[ \omega_\theta(v,v') = q(v,J_\theta v').\]
We have orthogonality $J_\theta^T = -J_\theta$ 
because $\omega_\theta$ is antisymmetric.
If $\theta\ne 0$ then $\omega_\theta$ is nondegenerate, which implies that $J_\theta$ is invertible and $J_\theta^2=-J_\theta^TJ_\theta<0$ is strictly negative. 
Since $\omega_\theta$ is compatible with $q$, it follows that $J_\theta^2$ is a negative scalar multiple of the identity.

The normalized trace
\[ \ang{\theta, \theta'} = \frac{1}{\text{\rm dim}\,V}\,\text{\rm trace}(J_\theta^T J_{\theta'}).\]
defines a positive definite quadratic form on $W$.
If $\ang{\theta, \theta}=1$ we obtain $J_\theta^2 = -J_\theta^T J_\theta = -1$.
This is the defining property of a Heisenberg type group.

\end{proof}

While Kaplan's remark in [Ka80] suggests that the class of polycontact groups is strictly larger than the class of Heisenberg type groups, it is not so easy to construct examples of polycontact groups that are not $H$-type.
While $H$-type groups are `rigid', in the sense that their isomorphism classes come in discrete families, the moduli space of polycontact groups is `wild' (they come in continuous families). 
This question is addressed in [LT99].

\subsection{Strictly Pseudoconvex Hypersurfaces in a Quaternionic Domain}

It seems that any reasonable generalization of contact structures to higher codimensions should at least include the unit sphere in quaternionic vectorspace.
For a point $a\in S^{4m+3}\subset \QT^{m+1}$ let $H_a$ be the largest linear subspace of $T_aS^{4m+3}$ that is invariant under the complex structures $I,J,K$ on $\QT^{4m+3}$.
It is well-known (see [Mo73]) that the manifold with distribution $(S^{4m+3}, H)$ is the one-point compactification of the quaternionic Heisenberg group $\QT^m\times \im \QT$ with its standard codimension 3 distribution,
just as the unit sphere in $\CC^{m+1}$ is contactomorphic to the Heisenberg group $\CC^m\times \RR$.
It follows that $S^{4m+3}$ is a $3$-polycontact manifold.

Denoting $(p,q)=(p_1,p_2,\cdots,p_m,q)\in \QT^{m+1}$ we can write an explicit $3$-contactomorphism as
\[ f\;\colon\; S^{4m+3}\setminus \{(0,-1)\}\to \QT^m\times \im\QT,\quad
f(p, q) = \left( (1+q)^{-1}p, \;\im (1+q)^{-1} \right).\]
The exercise to verify that this is a $3$-contactomorphism is not entirely trivial due to the noncommutativity of quaternion multiplication.
It is, however, rather easy to show that any ellipsoid in $\QT^{m+1}$ is a $3$-polycontact manifold. 
In fact, any strictly pseudoconvex real hypersurface in a quaternionic domain is $3$-polycontact.

\begin{proposition}
Let $M$ be a real codimension one submanifold in a hypercomplex manifold $(X, I, J, K)$
which is strictly pseudoconvex for any of the complex structures $L=aI+bJ+cK$ on $X$ with $a^2+b^2+c^2=1$.
Then the distribution $H$ which is the largest subspace of $TM$ that is invariant under the complex structures $I, J, K$ is a $3$-polycontact structure on $M$.
\end{proposition}
\begin{proof}
The Levi-form on $H$ for each of the complex structures $L=aI+bJ+cK$
is nondegenerate. 
But the Levi form for $L$ is just the bracket map $B\colon H\times H\to N$
composed with the functional $(a,b,c)$ in an appropriate basis for $N^*$.
Those functionals span $N^*$, and so the bracket $B$ satisfies the $3$-polycontact condition.
\end{proof}

The tangent groups of a strictly pseudoconvex hypersurface in a quaternionic domain are not necessarily isomorphic to the quaternionic Heisenberg group. 
In other words, these hypersurfaces provide examples of $3$-polycontact manifolds that are not quaternionic contact.

\subsection{Quaternionic Contact Manifolds and $3$-Sasakian Manifolds}\label{s2}

Geometers have experimented  with various definitions of contact structures in higher codimension, primarily (or almost exclusively) in codimension three.
As Geiges and Thomas explain in [GT95], these attempts were motivated by the existence of quaternionic analogs of complex and k\"ahler manifolds.
Hypercomplex, quaternionic k\"ahler, and hyperk\"ahler manifolds arise as Riemannian manifolds with special holonomy groups.
Since boundaries of complex manifolds may carry a contact structure,
so boundaries of hypercomplex manifolds should carry a $3$-contact structure of some sort.
The {\em hypercontact} structures of Geiges and Thomas, however, are not necessarily polycontact (although all the key examples of hypercontact manifolds listed in [GT95] are polycontact).

As is well-known, boundaries of rank one symmetric spaces carry a Carnot-Caratheodory metric, and their structure is closely related to specific Heisenberg type groups. 
For example, complex hyperbolic space $M=\CC H^{m+1}$ is a symmetric space with a hyperbolic k\"ahler metric.
The spherical boundary at infinity $\partial M = S^{2m+1}$ has a natural contact structure, induced by the asymptotics of the k\"ahler metric on $\CC H^{m+1}$.
Quaternionic hyperbolic space $M=\QT H^{m+1}$ is a symmetric space with a hyperbolic {\em hyperk\"ahler} metric. 
Just as in the complex case, the boundary sphere $\partial M=S^{4m+3}$ has a natural distribution $H$, which arises by the same asymptotic procedure that defines the contact structure on the boundary of complex hyperbolic space.
This is precisely the standard $3$-polycontact structure on $S^{4m+3}$ discussed above.
The same construction applies to the boundary of octonionic hyperbolic space, i.e., the Cayley plane.
Here we obtain a $7$-polycontact structure on $S^{15}$ isomorphic (up to a point) to the octonionic Heisenberg group $\OT\times \im\OT$.
See [Mo73] for details of the geometry involved in these well-known examples.

An early definition of a generalized contact structure is Ying-yan Kuo's notion of {\em almost contact $3$-structure} [Ku70]. 
Kuo's definition is rather loose and, in general, his ``almost'' contact $3$-structures are not polycontact.
A much more rigid sort of structure, also defined by Kuo, are the {\em 3-Sasakian manifolds}.
A $3$-Sasakian manifold is the boundary of a cone that is hyperk\"ahler (see, for example, [BGM93]).
\footnote{A manifold $M$ with Riemannian metric $g$ is $3$-Sasakian if the holonomy group on the metric cone $M\times \RR_{>0}$ with metric $t^2g+dt^2$ reduces to {\rm Sp}(m+1). In particular, $M\times \RR_{>0}$ is hyperk\"ahler.}
More recently, Olivier Biquard introduced {\em quaternionic contact structures} [Bi99],
loosening the notion of $3$-Sasakian manifolds so that more general  boundaries of hyperbolic hyperk\"ahler spaces are quaternionic contact manifolds.

Biquard's definition is as follows.
\begin{definition}
A {\em quaternionic contact structure} on a smooth manifold $M^{4m+3}$ is a distribution $H\subset TM$ of rorank $3$ equipped with a conformal class of metrics $[g]$; a (locally defined) triple of almost complex structures $J_1, J_2, J_3\in {\rm End}(H)$ satisfying the quaternionic relations $J_1^2=J_2^2=J_3^2=J_1J_2J_3=-1$; and a (locally defined) triple of $1$-forms $\eta=(\eta_1,\eta_2,\eta_3)$ such that $H={\rm ker}\eta$ and such that $\eta_i$ is compatible with the almost complex structure $J_i$,
\[ d\eta_i|_H = g(J_i\,\cdot, \cdot)\qquad i=1,2,3\]
\end{definition}
To understand the definition recall that a manifold with distribution is a contact manifold if and only if all tangent groups are isomorphic to the Heisenberg group $\CC^m\times\im\CC$.
Likewise, it is not hard to see that a manifold with distribution $H$ of corank $3$ is quaternionic contact if and only if all the tangent groups are isomorphic to the quaternionic Heisenberg group $\QT^m\times \im \QT$. 
However, the automorphism group of $\QT^m\times \im \QT$ is strictly larger than $CSp_mSp_1$, which is the structure group of a quaternionic contact distribution $H$. 
To reduce the structure group to $CSp_mSp_1$ (and thus induce a quaternionic contact structure) requires precisely the choice of a compatible conformal structure in $H$.
Thus, a conformal structure in the distribution $H$ is part of the quaternionic contact structure.
\footnote{On a contact manifold a compatible conformal structure in $H$ is equivalent to the choice of a compatible almost CR structure.}

Since the tangent groups are Heisenberg type groups, every quaternionic contact manifold (and, by implication, also every $3$-Sasakian manifold) is a $3$-polycontact manifold.
The converse is not true in general, because the quaternionic Heisenberg group is not the only $3$-polycontact group.
However, Proposition \ref{prop:conf}
implies that in the presence of a compatible conformal structure in $H$,
the tangent groups of a polycontact structure are necessarily Heisenberg type groups.
Let's call a polycontact structure $H\subset TM$ with compatible conformal structure a {\em conformal polycontact structure}.

Not all Heisenberg type groups with center of dimension $3$ are isomorphic to $\QT^m\times \im\QT$,
but there are only finitely many isomorphism classes of Heisenberg type groups of given dimension.
On a connected manifold with conformal polycontact structure all tangent groups are necessarily isomorphic,
and so there are only finitely many `types' of conformal polycontact structure in any given dimension.
It seems worth asking to what extent the other types of conformal $3$-polycontact manifolds (i.e., with tangent groups not isomorphic to $\QT^m\times \im\QT$)
are geometrically similar to quaternionic contact manifolds.

As is the case in many other respects, dimension $7$ is exceptional.

\begin{proposition}\label{prop4-7}
Every $3$-polycontact manifold of dimension $7$ is quaternionic contact.
\end{proposition}
\begin{proof}
It suffices to show that every polycontact group $G=\RR^4\times \RR^3$ with center of dimension $3$ is isomorphic to the quaternionic Heisenberg group $\QT\times \im\QT$.
To see this, consider the natural quadratic form on $\Lambda^2\RR^4$,
\[ q\;\colon\;\Lambda^2\RR^4\;\otimes\;\Lambda^2\RR^4\to \Lambda^4\RR^4 = \RR\;\colon\; q(\omega, \omega')=\omega\wedge \omega'.\]
This quadratic form is degenerate, with signature $(3,3)$.
A subspace $V\subset \Lambda^2\RR^4$ of dimension $3$ consists of nondegenerate $2$-forms
if and only if the quadratic form $q$ is nondegenerate on $V$.
We may assume it is positive definite.
The action of $SL(4,\RR)$ on $\RR^4$ induces an action on $\Lambda^2\RR^4$
that preserves the quadratic form. In fact, we have
\[ SL(4,\RR)\to SO(3,3),\]
and this map is two-to-one onto.
Thus, we can always pick an element $g\in SL(4,\RR)$ that maps any one subspace $V$ to any other such subspace $V'$. 
Then $g$ implements an isomorphism between the corresponding nilpotent groups.

\end{proof}


\subsection{Fat Bundles}

Fat bundles were introduced by Alan Weinstein 
in connection with the study of Riemannian fibre bundles with totally geodesic fibres and positive sectional curvatures [We80].
Let $G\to P\to M$ be a smooth principal $G$-bundle with connection $1$-form $\omega$.
Recall that curvature $\Omega=d\omega + \frac{1}{2}[\omega,\omega]$ is a horizontal $\Lg$-valued $2$-form on $P$. 
By composing with a functional $\mu\in \Lg^*$ we obtain a horizontal scalar valued $2$-form $\mu\circ \Omega$  on $P$.

Let $H\subset TP$ denote the horizontal bundle, i.e., the kernel of the connection $1$-form $\omega$.
If the restriction of $\mu\circ \Omega$ to $H$ is zero for every $\mu\in \Lg^*$ then the connection $\omega$ is flat.
By contrast, $\omega$ is {\em fat} if the restriction of $\mu\circ \Omega$ to $H$ is nondegenerate for each non-zero $\mu\in \Lg^*\setminus 0$.
Since $\Omega=d\omega$ when restricted to $H$, 
it is immediate from the definitions that $\omega$ is flat iff $H$ is a foliation on $P$, while $\omega$ is fat if and only if $H$ defines a polycontact structure.
(The example of fat bundles was pointed out to me by Sean Fitzpatrick [Fi11].)

A fat principal $S^1$-bundle is the same as a regular contact manifold.
In general, a fat bundle is the same as a polycontact manifold with a transversal principal $G$-action preserving the polycontact distribution.
Various examples of fat bundles (and hence of polycontact manifolds) are discussed in [We80]. In particular, a result of B\'erard Bergery, quoted in [We80], implies that most of the known examples of positively curved manifolds can be given a polycontact structure (at least as of 1980...).

\subsection{CR Structures}

Contact structures arise naturally on the boundary of complex domains.
Let $M=\partial \Omega$ be the boundary of a domain $\Omega$ in a complex manifold. 
Let $H^{0,1}$ be the complex vector bundle on $M$ that is the intersection of the complexified tangent space $TM\otimes \CC$ with the bundle of antiholomorphic vectors $T^{0,1}\Omega$ restricted to $M$. 
The complex bundle $H^{0,1}$ is involutive, in the sense that $[H^{0,1}, H^{0,1}]\subseteq H^{0,1}$,
and is called a {\em CR structure} on $M$.
If the domain $\Omega$ is strictly pseudoconvex, or more generally if the Levi form is nondegenerate, then the real bundle $H$ underlying $H^{0,1}$ is a contact structure on $M$.

In general, a CR structure is a subbundle $H^{0,1}\subseteq TM\otimes \CC$ that is involutive, and such that its complex conjugate $H^{1,0}$ intersects $H^{0,1}$ only in the zero vectors.
The codimension $p$ of such a CR structure can be strictly greater than one.
By definition, the real bundle $H$ underlying a CR structure is polycontact if and only if  the Levi forms $\theta\circ B_m$ are nondegenerate for all $m\in M$, $\theta\in N^*\setminus 0$.
In the context of CR geometry the term `Levi form' usually refers to a {\em hermitian} form on $H^{1,0}$.
The hermitian Levi form on $H^{1,0}$ is the standard hermitian form corresponding to the anti-symmetric bracket $B_m$ that we call `Levi form' here.
More precisely, we identify the bundle $H$ with its complex structure $J\in \mathrm{ End}(H)$ with the complex vector bundle $H^{1,0}$ in the usual way.
If $B_m\,\colon H_m\times H_m\to N_m$ is the bracket that defines the tangent group, as above,
we obtain hermitian forms on $H=H^{1,0}$ by
\[ 
h_\theta(v,w) = \theta[B_m(v,Jw)+iB_m(v,w)]\qquad v,w\in H_m,\; \theta\in N^*_m\setminus 0
\]
Then $H$ is polycontact if and only if all the (hermitian) Levi forms $h_\theta$ are nondegenerate.

In codimension $p=1$ strictly pseudo-convex CR structures are of special interest.
A CR structure of codimension $1$ is strictly pseudoconvex if its (hermitian) Levi form is positive-definite.
Since the sphere $SN^*_m=S^{p-1}$ is connected if $p>1$,
on a polycontact CR structure of codimension $p>1$ all nondegenerate hermitian forms $h_\theta$ must have the same signature.
But because $h_{-\theta}=-h_\theta$ the signature of the Levi forms must then necessarily of type $(n,n)$. 
Thus, while `positive-definiteness' of the Levi forms $h_\theta$ is not possible in codimensions $p>1$,
the polycontact condition {\em can} be satisfied if the Levi forms have split signature.
For example, the Heisenberg type groups $\mathcal{C}(2,n)=\CC^{2n}\times \CC$ are examples of polycontact CR structures of (real) codimension $p=2$,
with Levi forms of signature $(n,n)$.

\section{The Heisenberg Calculus on Polycontact Manifolds}\label{section:Heis}

The Heisenberg calculus is the natural pseudodifferential calculus associated to the pair $(M,H)$. 
It was introduced by Taylor [Ta84] for contact manifolds, developing ideas of Folland and Stein related to the analysis of operators associated to CR structures [FS74].
Beals and Greiner [BG88] develop the calculus for arbitrary distributions of codimension $1$. 
As pointed out to us by the referee, the results in [BG88] hold verbatim for arbitrary distributions of codimension $\ge 2$.
Other good references for the Heisenberg calculus are [EM03, CGGP92, Po08].

\subsection{Brief Overview of the Heisenberg Calculus}

Like any pseudodifferential calculus, the Heisenberg calculus consists of linear operators 
\[ P\;\colon\; C^\infty(M)\to C^\infty(M) ,\]
whose Schwartz kernel $K\in \DD'(M\times M)$ is smooth off the diagonal in $M\times M$, and has an asymptotic expansion $K\sim \sum K^j$ near the diagonal.
The terms $K^j$ in the expansion are homogeneous in a suitable sense (see below),
and the highest order part $K^0$ can be interpreted as a family of convolution operators on the nilpotent tangent groups.

We shall denote by $\Psi^d_H(M)$ the space of Heisenberg pseudodifferential operators of order $d$ on $M$,
and by ${\mathcal S}_H^d(M)$ the space of homogeneous Heisenberg symbols of degree $d$.
Recall that if $\Lg^*=H^*\oplus N^*$ is the dual of the Lie algebra bundle for the the tangent groups, then ${\mathcal S}_H^d(M)$ consists of smooth functions on $\Lg^*\setminus 0$ that are homogeneous of degree $d$ with respect to the natural dilations $\delta_s$ on $\Lg^*_m, m\in M$ defined by
\[ \delta_s(\xi,\theta) = (s\xi,s^2\theta),\quad \xi\in H^*_m,\; \theta\in N^*_m,\; s>0 .\]
The highest order part in the expansion of the Schwartz kernel for an operator $P\in \Psi_H^d(M)$ is a well-defined distribution on the bundle of tangent groups $H\oplus N=\sqcup_{m\in M} G_m$.
The Fourier transform of this distribution (in each fiber) is the principal symbol $\sigma^d_H(P)\in {\mathcal S}^d_H(M)$. 
The full symbolic calculus gives an asymptotic expansion of symbols of the product of two Heisenberg pseudodifferential operators,
and the principal parts multiply according to a suitable product law for homogeneous symbols,
\[ \#\;\colon\;{\mathcal S}^a_H(M)\times {\mathcal S}^b_H(M) \to {\mathcal S}^{a+b}_H(M).\]
The $\#$ product is dual to convolution of homogeneous kernels on the tangent group $G_m=H_m\times N_m$,
and so it is not the usual pointwise product of functions.
If the tangent group is nonabelian the $\#$ product is noncommutative.
However, when we restrict symbols to $H^*\setminus 0\subset \Lg^*\setminus 0$ then the product agrees with pointwise multiplication of functions,
\[ (\sigma_1\# \sigma_2)(m,\xi) = \sigma_1(m,\xi)\sigma_2(m,\xi),\quad m\in M,\,\xi\in H^*_m\setminus 0,\, \sigma_j\in {\mathcal S}_H^{d_j}(M).\]
Hermite operators $I^d_H(M)\subset \Psi^d_H(M)$ are operators whose symbol vanishes on $H^*\setminus 0$.
We denote the algebra of Hermite symbols of order zero as
\[ S^{00}_H(M) = \{\sigma\in S^0_H(M)\;\mid\; \sigma(m,\xi)=0\;\; \forall (m,\xi)\in H^*\setminus 0\}.\] 
The terminology `Hermite ideal' is taken from [EM03] and the notation $S^{00}_H(M)$ is borrowed from [Ta84].
The relevance of the Hermite ideal for our purposes is due to the following lemma.

\begin{lemma}\label{lemma:one}
Let $M$ be a connected closed manifold with distribution 
$H\subseteq TM$.
An idempotent in the Heisenberg calculus $\Psi^\bullet_H(M)$
is either a Hermite operator of order zero, 
or it is of the form $1-x$ 
where $x$ is an idempotent that is a Hermite operator of order zero.
\end{lemma}
\begin{proof}
Suppose $P\in \Psi^d_H(M)$ and $P^2=P$.
Then $P$ is necessarily of order zero (or of order $-\infty$, in which case it is also order zero).
Moreover, the restriction of the symbol $\sigma^0_H(P)(m,\xi)$ to $H^*\setminus 0$ is an idempotent function, i.e., it has range $\{0,1\}$.
Since $H^*\setminus 0$ is connected, either $\sigma^0_H(P)(m,\xi)=0$ for all $(m,\xi)$, in which case $P\in I_H^0(M)$, or $\sigma^0_H(P)(m,\xi)=1$
in which case $P=1-x$ with $x\in I_H^0(M)$.
The lemma follows since $(1-x)^2=1-x$ implies $x=x^2$.

\end{proof}

\subsection{Idempotents}

Let us denote by $\sigma^\vee(m,h,\theta)$ the inverse Fourier transform of $\sigma(m,\xi,\theta)$ with respect to the variable $\xi\in H^*$.
We will write $\sigma^\vee(m,\theta)\in \mathcal{S}(H_m)$ for $m\in M, \theta\in N^*\setminus 0$.
Regardless of the geometric nature of the distribution $H$, 
the product of symbols $\sigma^\vee(m,\theta)$ in the Heisenberg calculus can be expressed as `twisted convolution' of functions in the Schwartz class $\mathcal{S}(H_m)$.
In general, convolution twisted by a group cocycle 
$\omega\;\colon \RR^n\times \RR^n\to U(1)$ is given by
\[ (f\ast_\omega g)(x) = \int \omega(x,y)\, f(x-y)g(y) \,dy,\quad f,g\in\mathcal{S}(\RR^n) \]
We can then write succinctly
\[ (\sigma_1\# \sigma_2)^\vee(m,\theta) = (\sigma_1)^\vee(m,\theta) \ast_\omega (\sigma_2)^\vee(m,\theta).\] 
where the convolution is twisted by the cocycle 
\[ \omega_{m,\theta}(h,h')= \exp{(-\frac{i}{2}\theta\circ B_m(h,h'))}.\]
The Schwartz class $\mathcal{S}(H_m)$ is dense in the twisted convolution $C^*$-algebra $C^*(H_m, B_\theta)$, which is a well-known $C^*$-algebra whose structure is easy to describe.
As usual, the $C^*$-algebra of compact operators on separable Hilbert space is denoted $\KK$.

\begin{lemma}\label{lemma:three}
Let $V\cong \RR^p$ be a real vector space,
and $B\colon \Lambda^2 V\to \RR$ a two-form on $V$.
Let $W$ denote the kernel of $B$, i.e., the set of vectors $w$ for which $B(v, w) = 0$ for all  $v\in V$.
If $B$ is not zero, then the twisted convolution algebra $C^*(V, B)$ is isomorphic to $C_0(W^*, \KK)$.
In the degenerate case where $B=0$ and so $V=W$ we have simply $C^*(V, B)\cong C_0(V^*)$. 
\end{lemma}   
\begin{proof}
If $B$ is symplectic then $W=\{0\}$ and the twisted convolution algbera $C^*(V, B)$ is the well-known Weyl-algebra, which in its regular representation on $L^2(V)$
corresponds to the algebra of compact operators.
In general, the quotient space $V/W$ is symplectic.
Denote the symplectic form on $V/W$ by $\omega$. 
Choose a complement $V=W\oplus W'$, then $W'$ inherits the symplectic form $\omega$.
The Fourier transform in $W$ then gives an isomorhism
\[ C^*(V, B) \cong C_0(W^*, C^*(W',\omega)).\]
But now $C^*(W',\omega))\cong \KK$, assuming $W'$ is not zero.
The degenerate case $B=0$, on the other hand, is clear.
\end{proof}

\begin{corollary}\label{cor:three}
The $C^*$-algebra $C^*(H_m,B_\theta)$ contains a nonzero idempotent if and only if $B_\theta$ is nondegenerate.
\end{corollary}
\begin{proof} Idempotents have norm $1$, and so $C_0(W^*, \KK)$ cannot contain any idempotents unless $W=\{0\}$. 
\end{proof} 

Any family of convolution $C^*$-algebras $C^*(\RR^q, \omega_x)$ with continuously varying $U(1)$-valued cocycles $\omega_x$ ,$x\in \RR^n$, constitutes a continuous field over $\RR^n$ [Ri89].
For our purposes, the main import of this fact is that the $C^*$-norms $\|\sigma^\vee(m,\theta)\|$ depend continuously on $(m,\theta)\in S^*N$,
which one can also verify directly without much trouble.

\begin{proposition}\label{prop:one}
Let $(M,H)$ be a connected manifold with distribution.
If the Heisenberg algebra contains an idempotent with infinite dimensional kernel and range, then $M$ is polycontact.
\end{proposition}
\begin{proof}
By Lemma \ref{lemma:one} it suffices to consider the Hermite ideal $I^0_H$.
Suppose $P\in I^0_H$ is a projector with infinite dimensional kernel and range.
Then the symbol $\sigma=\sigma_H(P)$ of $P$ is a nonzero projector
(or else $P$ would of order $-1$, and hence compact).
Since each $\sigma^\vee(m,\theta)$ is an idempotent in $C^*(H_m,B_\theta)$ the $C^*$-norms $\|\sigma^\vee(m,\theta)\|$ are either $0$ or $1$.
Since the norm $\|\sigma^\vee(m,\theta)\|$ is continuous in $(m,\theta)$ it is a locally constant function.
In corank $p>1$ the sphere bundle $S^*N$ is connected, and it follows that $\sigma^\vee(m,\theta)$ is nonzero for every $(m,\theta)$.
By Corollary \ref{cor:three} this implies that $\theta\circ B_m$ is nondegenerate for all $(m,\theta)$, which means that $M$ is polycontact.

In corank $p=1$ the normal bundle $N$ has one dimensional fiber, and the sphere bundle $S^*N$ may not be connected. 
Nevertheless, since $M$ is connected it still follows that either $B_\theta$ or $B_{-\theta}$ is nondegenerate for every $m\in M$, which, again, implies that the manifold is contact. 

\end{proof}

\subsection{The Structure of the Heisenberg $C^*$-algebra}

To complete our analysis we pass to the norm-completion of the  Heisenberg algebra $\Psi_H^0(M)$.
Order zero operators are bounded on $L^2(M)$,
and the completion of $\Psi_H^0(M)$ in operator norm is a $C^*$-algebra
that we denote by $\Psi^*_H$. 
In general, the $C^*$-algebra $\Psi^*_H$
is a Type I (even CCR) algebra with a finite decomposition series.
In [Dy78], Alexander Dynin announced the result that $\Psi_H^*$ has a decomposition series of length $2$ if $M$ is a contact manifold (codimension one).
Our analysis in the previous section implies that this is true for all polycontact manifolds.

For compact $M$ with distribution $H$ there is the obvious short exact sequence of $C^*$-algebras
\[ 0\to I^*_H \to \Psi^*_H \to C(S^*H) \to 0,\]
where $I^*_H$, $\Psi^*_H$ denote the completions of the algebras $I^0_H(M)$, $\Psi^0_H(M)$ respectively.
We need to decompose $I_H^*$ using the symbol map.
For polycontact manifolds the norm completion of the algebra of Hermite symbols $S^{00}_H(M)$ has a simple description.

\begin{proposition}\label{prop:calc}
Let $(M, H)$ be a compact polycontact manifold.
Then the $C^*$-algebraic completion of the algebra of order zero Hermite symbols $S^{00}_H(M)$, i.e., the quotient $I^*_H/\KK$, is isomorphic to $C(S^*N, \KK)$, the algebra of continuous functions from $S^*N$ to $\KK$. 
\end{proposition}
\begin{proof}
Let $S^*N$ be the unit sphere bundle of $N^*$ for an arbitrary Euclidean structure.
Let $\pi^*H\to S^*N$ denote the pull-back of the vector bundle $H\to M$ along $\pi\;\colon S^*N\to M$.
The vector bundle $\pi^*H$ has a canonical symplectic structure, with symplectic forms $\theta\circ B_m$ in each fiber $H_{m,\theta}$.
Let $J\in {\rm End}(\pi^*H)$ be a compatible complex structure, 
and let $H^{BF}_{m,\theta}$ be the associated field of Bargmann-Fok Hilbert spaces over $S^*N$.
This field is constructed in exactle the same way as on contact manifolds (see [EM03]).
The Bargmann-Fok representations $\pi_{m,\theta}$ of the algebra of Hermite symbols induce canonical isomorphisms
\[ \pi_{m,\theta}\;\colon\; C^*(H_m, B_\theta) \cong \KK(H^{BF}_{m,\theta}).\]
Therefore the completion $I^*_H/\KK$ of $S^{00}_H(M)$ can be identified with the $C^*$-algebra of sections in the continuous field over $S^*N$ with fibers $\KK(H^{BF}_{m,\theta})$.
But every field of (infinite dimensional) Hilbert spaces over a manifold is isomorphic to a trivial field ([Di64], Lemma 10.8.7), which proves the claim.

\end{proof}

We see that for every compact polycontact manifold we have a decomposition series consisting of only two sequences
\[ 0\to I_H^*\to \Psi_H^*\to C(S^*H)\to 0\]
\[ 0\to \KK\to I_H^*\to C(S^*N, \KK)\to 0\]
generalizing Dynin's result for contact distributions of codimension one.
One easily checks that if $H$ is not polycontact then the decomposition series for the Hermite ideal $I_H^*$ requires more than one sequence.
The geometric condition of ``maximal non-integrability'' of the distribution $H$ is thus {\em equivalent} to a precise structural feature of the Heisenberg $C^*$-algebra.

\begin{corollary}\label{cor:two}
Let $(M, H)$ be a compact polycontact manifold.
Then there exist idempotents in the Hermite ideal $I_H^0(M)$ with infinite dimensional kernel and range.
\end{corollary}
\begin{proof}
Clearly there are nonzero projections in $I_H^*/\KK\cong C(S^*N, \KK)$,
and therefore in $I_H^*$.
Since the pseudodifferential algebra $\Psi^0_H(M)$ is closed under holomorphic functional calculus it follows that non-compact projections exist in the dense subalgebra $I_H^0(M)\subset I^*_H$. 
\footnote{The fact that the Heisenberg pseudodifferential algebra is closed under holomorphic functional calculus follows from results in [Po08].
We thank the referee for the reference.}
\end{proof}

\subsection{Generalized Szeg\"o Projections}

Because of the strong structural similarities of the Heisenberg algebra of polycontact and contact manifolds, the construction of generalized Szeg\"o projections on contact manifolds of Boutet de Monvel and Guillemin [BG81] or of Epstein and Melrose [EM03] carries over without significant change to polycontact manifolds.
The main import of Theorem \ref{thm:main} is perhaps that ``maximal non-integrability'' of the distribution $H$ is equivalent to the existence of such generalized Szeg\"o projections.

As we have shown, the norm completion of the algebra of Hermite symbols $S^{00}_H(M)$ is naturally identified with a continuous field over the compact space $S^*N$ with fibers $C^*(H_m,B_\theta)$.
As above, let $J$ be a complex structure in the pull-back vector bundle $\pi^*H\to S^*N$ compatible with the canonical symplectic forms $\theta\circ B_m$ for each $(m,\theta)\in S^*N$.
Then an explicit projection in the twisted convolution algebra $C^*(H_m,B_\theta)$ is given by the formula
\[ \sigma_0^\vee(m,h,\theta) = \frac{1}{(2\pi)^n} \exp{(-\frac{1}{4}\theta\circ B_m(h,Jh)))}\qquad m\in M,\, h\in H_m,\, \theta\in S(N^*)\]
(In the Bargmann-Fok representation of $C^*(H_m,B_\theta)$ this is the projection onto the vacuum vector.)
Let $g(m,\theta)\ang{-,-}$ denote the Euclidean inner product in $H^*_m$ dual to $\theta(B_m(-,J-))$.
Observe that $g(m,t\theta)=t^{-1}g(m,\theta)$.
The $h\mapsto \xi$ Fourier transform of $\sigma_0^\vee$ defines an idempotent symbol $\sigma_0\in S^{00}_H(M)$, 
\[ \sigma_0(m,\xi,\theta) =
\begin{cases} 2^n \exp{(-g(m,\theta)\ang{\xi,\xi})} &\theta\ne 0 \\
0 &\theta=0
\end{cases} 
\]
A generalized Szeg\"o projection $S_0\in \Psi^0(M)$
with symbol $\sigma_0$ is constructed following the usual procedure, which, just as our proof of Corollary \ref{cor:two},  relies on the closure under holomorphic functional calculus of $\Psi_H^0(M)$.

One can similarly generalize the Szeg\"o projections of ``higher levels'' defined in [EM03] to polycontact manifolds.

\vskip 6pt
\noindent {\bf Remark.}
According to Proposition \ref{prop:conf}, if all the tangent groups are of Heisenberg type, then there exists a conformal structure for the distribution $H$ that is compatible with the polycontact structure.
For example, a strictly pseudoconvex CR structure of codimension $1$ gives rise to a contact structure with compatible conformal structure.
Likewise, on a quaternionic contact manifold one often considers a compatible conformal structure in $H$.

We have seen that such a conformal structure, if it exists, induces a compatible  conformal structure in the normal bundle $N$.
Given such a conformal structure we then simply have
\[ \sigma_0(m,\xi,\theta) =
\begin{cases} 2^n \exp{(-\|\xi\|^2\,/\,\|\theta\|)} &\theta\ne 0 \\
0 &\theta=0
\end{cases} 
\]
and the Szeg\"o projection is uniquely determined up to a compact perturbation.

\section*{References}

\def\item{\vskip2.75pt
plus1.375pt minus.6875pt\noindent\hangindent1em} \hbadness2500
\tolerance 2500
\markboth{References}{References}

\item{[Ba02]} J.\ Baez,
{\sl The octonions},
Bull.\ Amer.\ Math.\ Soc.\ 39 (2002), 145--205

\item{[BG88]} R.\ Beals and P.\ Greiner,
{\sl Calculus on Heisenberg Manifolds},
Annals of Mathematics Studies (119), Princeton, 1988.

\item [BG81] L.\ Boutet de Monvel and V.\ Guillemin,
{\sl The spectral theory of Toeplitz operators},
Ann.\ Math.\ Studies 99, Princeton Univ.\ Press, 1981.

\item [BGM93] 
C.\ Boyer, K.\ Galicki, B.\ Mann, {\sl $3$-Sasakian manifolds},
Proc.\ Japan Acad.\ Ser.\ A Math.\ Sci.\ 69 (1993), no. 8, 335--340.

\item [Bi99] O.\ Biquard, {\sl Quaternionic contact structures},
Quaternionic structures in mathematics and physics, Univ. Studi Roma ``La Sapienza'', Rome, 1999.

\item{[CDKR91]} M.\ Cowling, A.\ H.\ Dooley, A.\ Kor\'anyi, and F.\ Ricci,
{\sl $H$-type groups and Iwasawa decompositions},
Adv.\ Math.\, 87, 1--41  (1991)

\item{[CGGP92]}, M.\ Christ, D.\ Geller, P.\ Glowacki, L.\ Polin,
{\sl Pseudodifferential operators on Groups with dilations},
Duke Math.\ J\. 68, 31--65  (1992)

\item{[CS03]} J.\ H.\ Conway and D.\ A.\ Smith,
{\it On Quaternions and Octonians: their geometry, arithmetic, and symmetry},
A.\ K.\ Peters, 2003.

\item{[Di64]} J.\ Dixmier, 
{\sl $C^*$-alg\`ebres et leurs repr\'esentations},
Gauthier-Villars and Cie, Paris, 1964.

\item [Dy78] A.\ Dynin,
{\sl Inversion problem for singular integral operators: $C^*$-approach},
Proc.\ Natl.\ Acad.\ Sci.\ USE, Vol. 75, No. 10, October 1978, 4668--4670.

\item{[EM03]} C.\ Epstein and R.\ Melrose,
{\sl The Heisenberg algebra, index theory and homology},
preprint, 2003.

\item {[Fi11]} S.\ Fitzpatrick, {\sl private communication}, 2011

\item{[FS74]} G.\ B.\ Folland and E.\ M.\ Stein,
{\sl Estimates for the $\bar{\partial}_b$ complex and analysis on the Heisenberg group},
Comm.\ Pure and Appl.\ Math.\, vol XXVII (1974), 429--522.


\item{[GT95]} H.\ Geiges and C.\ B.\ Thomas,
{Hypercontact manifolds},
J.\ London Math.\ Soc.\ (2)  51  (1995),  no. 2, 342--352.



\item {[Hu66]} D.\ Husemoller, {\sl Fibre bundles}, McGraw-Hill, 1966

\item [IV] S.\ Ivanov and D.\ Vassilev,
{\sl Conformal quaternionic contact curvature and the local sphere theorem},
preprint, arXiv:0707.1289v2 [math.DG]


\item {[Ka80]}
A.\ Kaplan,
{\sl Fundamental solutions for a class of hypoelliptic PDE generated by composition of quadratic forms},
Trans.\ Amer. Math.\ Soc.\ 258, nr. 1 (1980)

\item{[Ka83]} A.\ Kaplan,
{\sl The geometry of groups of Heisenberg type},
Bull.\ London Math.\ Soc., 35--42 (1983)

\item [Ku70]
Y.\ Kuo, {\sl On almost contact $3$-structure},
T\^ohoku Math.\ J.\ (2), 22, 325--332 (1970)

\item [LM89]
H.\ Lawson and M.\ Michelson,
{\sl  Spin Geometry},
Princeton University Press, Princeton, 1989. 
 
\item{[LT99]} F.\ Levstein and A.\ Tiraboschi, 
{\sl Classes of 2-step nilpotent Lie algebras}, 
Comm.\ in Algebra 27, 2425--2440  (1999)

\item[Mo73]
G.\ D.\ Mostow, {\sl Strong rigidity of locally symmetric spaces},
Annals of Math.\ Studies 78, Princeton Univ.\ Press, 1973.  


\item {[Po08]} R.\ Ponge,
{\sl Heisenberg calculus and spectral theory of hypoelliptic operators on Heisenberg manifolds},
Mem.\ Amer.\ Math.\ Soc.\, vol. 194, nr. 906, AMS, 2008.

\item {[Ri89]}
M.\ Rieffel,
{\sl Continuous fields of $C^*$-algebras coming from group cocycles and actions},
Math.\ Ann.\ 283, 631--643 (1989)

\item [Sm94] H.\ Smith,
{\sl A calculus for three dimensional CR manifolds of finite type},
J.\ Funct.\ Anal.\ 120, 135--162 (1994)

\item{[Ta84]} M.\ E.\ Taylor,
{\sl Noncommutative microlocal analysis, part I},
Mem.\ Amer.\ Math.\ Soc.\, vol. 313, AMS, 1984.

\item{[We80]} A.\ Weinstein,
{\sl Fat bundles and symplectic manifolds},
Adv.\ Math.\, 37, 239--250  (1980)

\newpage

\end{document}